\documentclass{ipart}

\usepackage{amsthm,amsmath,graphicx, amssymb}
\usepackage[margin=4.5cm]{geometry}

\firstpage{1}
\lastpage{1}

\startlocaldefs
\theoremstyle{plain}
\newtheorem{thm}{Theorem}[section]
\newtheorem{lemma}[thm]{Lemma}

\newtheorem*{thm*}{Theorem}
\newtheorem*{lemma*}{Lemma}
\newtheorem*{prop*}{Proposition}
\newtheorem*{cor*}{Corollary}
\newtheorem*{conj*}{Conjecture}

\theoremstyle{definition}
\newtheorem{defn}[thm]{Definition}
\newtheorem{ex}[thm]{Example}

\newtheorem{alg}[thm]{Algorithm}

\theoremstyle{remark}
\newtheorem*{rmk}{Remark}

\DeclareMathOperator{\rank}{rank}

\DeclareMathOperator{\im}{Im}

\endlocaldefs

\begin{document}

\begin{frontmatter}

\title{A New Shellability Proof of \\  an Old Identity of Dixon\thanksref{t1}}
\runtitle{A New Proof of Dixon's Identity}



\begin{aug}

\author{\fnms{Ruth}
\snm{Davidson,}\thanksref{t1,t2}\ead[label=e1]{redavid2@illinois.edu}}
\address{Ruth Davidson \\ Department of Mathematics \\  University of Illinois Urbana-Champaign \\1409 W. Green Street, Urbana, IL \\  U.S.A. 61801 \\
\printead{e1}\\[3ex] \\}

\and

\author{\fnms{Augustine} \snm{O'Keefe,}\ead[label=e2]{augustine.okeefe@conncoll.edu}}
\address{Augustine O'Keefe \\ Department of Mathematics \\ Connecticut College\\ 270 Mohegan Ave.
New London, CT \\  U.S.A. 06320
\\
\printead{e2}\\ \\ }

\and

\author{\fnms{Daniel} \snm{Parry}
\ead[label=e3]{dan.t.parry@gmail.com}}

\address{Daniel Parry \\ Credit Suisse, Inc. \\ 11 Madison Ave, New York, NY \\  U.S.A. 10010
\\
\printead{e3}\\ }

\thankstext{t1}{Supported by U.S. N.S.F. grants DMS-0954865 and DMS-1401591}
\thankstext{t2}{Corresponding author}
\runauthor{Davidson, O'Keefe and Parry}

\end{aug}

\runauthor{Davidson, O'Keefe, and Parry}

\begin{abstract}
We give a new proof of an old identity of Dixon (1865-1936) that uses tools from topological combinatorics.  Dixon's identity is re-established by constructing an infinite family of non-pure simplicial complexes $\Delta(n)$, indexed by the positive integers, such that the alternating sum of the numbers of faces of $\Delta(n)$ of each dimension is the left-hand side of the identity.  We show that $\Delta(n)$ is shellable for all $n$. Then, using the fact that a shellable simplicial complex is homotopy equivalent to a wedge of spheres, we compute the Betti numbers of $\Delta(n)$ by counting (via a generating function) the number of facets of $\Delta(n)$ of each dimension that attach along their entire boundary in the shelling order. In other words, Dixon's identity is re-established by using the Euler-Poincar\'{e} relation. 

\end{abstract}



\end{frontmatter}

\section{Introduction}\label{introduction}

In this manuscript we give a new proof of the identity

\begin{equation}\label{DixID} \sum_{s = 0}^{n} (-1)^{s} { n \choose s}^{3}  = \left\{\begin{array}{lr} 0  & \text{if $n$ is odd, and } \\  (-1)^{n/2} {3n/2 \choose n/2,n/2,n/2}, & \text{if $n$ is even. } \end{array} \right\}\end{equation}

using tools from topological combinatorics.  To our knowledge, this is the first instance of using such tools to study an identity involving alternating sums of binomial coefficients. We hope this approach may illuminate past difficulties that have arisen in resolving ``closed-form" descriptions of identities involving powers of binomial coefficients in other enumerative disciplines (see Sections \ref{sec_discussion} and \ref{sec_futurework} for a discussion of these issues).

The identity (\ref{DixID}) is first attributed to Alfred Cardew Dixon (1865-1936). Dixon originally proved (\ref{DixID}) in the paper  ``On the sum of the cubes of the coefficients in a certain expansion by the binomial theorem" in \emph{Messenger of Mathematics} Volume 20 \cite{DixonReference}, which is a journal that ceased to publish in 1929.  See page 121 \cite{MacMahon} for an early application of (\ref{DixID}) that establishes the earlier reference.

The identity (1) is actually a special case of a more general identity.  Let $n_{1}, n_{2}$, and $n_{3}$ be nonnegative integers and let $N = n_{1} + n_{2} + n_{3}$.  Then  

\begin{equation}\label{3F2} \begin{split} \sum_{s = \max(0, n_{1}-n_{2}, n_{3}-n_{2})}^{ \min(n_{1}, n_{3}, n_{1} + n_{3}-n_{2})} {n_{3} \choose s} {n_{2} \choose n_{1} -s} {n_{1} \choose n_{2}-n_{3} + s} (-1)^{s} = \\ 
\begin{cases} 0  \  \text{if $N $ is odd,} \\ (-1)^{N/2-n_{2}} {N/2 \choose N/2-n_{1}, N/2 - n_{2}, N/2 - n_{3}} \  \text{if $N$ is even.}  \end{cases} \end{split} \end{equation} 

\noindent The identity (\ref{DixID}) is the case $n_{1} = n_{2} = n_{3} = n$.  The general case (\ref{3F2}) is called the \emph{well-poised $\mathstrut_3 F_2$ transformation}, and so is an example of a hypergeometric identity (see page 97 of \cite{Bailey}). See the proof of Lemma 4.2 in the paper \cite{CharneyDavis} for an example of an application of the identity (\ref{3F2}) in topological combinatorics. We have not found a topological interpretation of the identity \ref{3F2}, but we note the generalization because we believe the merit of studying (\ref{DixID}) arises from the connection it represents between enumeration techniques in number-theoretic algebraic combinatorics, topological combinatorics, and the general study of identities.

\section{Background and Definitions}

Establishing (\ref{DixID}) using a generating function is a relatively simple exercise (see, for example, page 23 of \cite{masterthoerem}). The novel contribution of this manuscript is the connection between (\ref{DixID}) and the combinatorial properties of a topological space, in particular a shellable simplicial complex.  Therefore, we begin with a brief introduction to shellability and the necessary background for shellable simplicial complexes. 

The notion of shellability originated in polyhedral theory via the study of boundary complexes of convex polytopes: 

\begin{thm}\label{boundary_thm}
The boundary complex of a convex polytope is shellable.
\end{thm}

Shl\"{a}fli assumed Theorem \ref{boundary_thm} in the nineteenth century when he computed the Euler characteristic of a convex polytope \cite{Schlafli} (see Lecture 3 of \cite{GeoCombiWachs} for a nice historical discussion of this), but Theorem \ref{boundary_thm} was not proved until 1970 \cite{BruggesserMani}, and was soon used in important results such as the proof of the Upper Bound Theorem for simplicial polytopes \cite{McMullen}.  In this manuscript, we will only define shellability for simplicial complexes.   

Shellability helps one to understand the structure of a simplicial complex via its topological and combinatorial qualities.   However, there are other properties of simplicial complexes with similar utility that have a long history in the literature and remain active areas of study.  These include partitionability \cite{StanleyBalanced}, collapsibility \cite{Whitehead}, and contractibility \cite{Borsuk}. Furthermore, relationships between such properties are still being resolved in recent papers such as \cite{NonPartionableCM}.  Topological properties of simplicial complexes also play a role in many fields of applied mathematics, notably topological data analysis. 

\begin{defn}\label{simplicial_complex_defn}

An \emph{(abstract) simplicial complex} on a vertex set $V$ is a collection $\Delta$ of subsets of $V$  satisfying 

\begin{enumerate}

\item if $v \in V$ then $\{ v \} \in \Delta$, and 

\item if $F \in \Delta$ and $G \subseteq F$, then $G \in \Delta$.  

\end{enumerate}

\end{defn}

The subsets of $V$ comprising $\Delta$ are called \emph{faces} or \emph{simplices}.  The dimension $\dim F$ of a face $F$ is $|F| - 1$, and $\dim \Delta$ is simply $\max \{ \dim F : F \in \Delta \}$.   A face $F$ is a \emph{facet} if $F$ is not properly contained in any other face of $\Delta$.  We say $\Delta$ is \emph{pure} if all the facets of $\Delta$ have the same dimension.  We write $\overline{F}$ to denote the sub-complex of $\Delta$ generated by $F$, or in other words $\overline{F} = \{ G  \in \Delta  : G \subseteq F \}$.  

\begin{defn}\label{shellability_defn}
A  simplicial complex $\Delta$ is \emph{shellable} if its facets can be arranged in a linear order $F_{1}, \ldots , F_{t}$ so that the subcomplex $\left( \bigcup_{i = 1}^{k-1} \overline{F}_{i} \right) \cap \overline{F_{k}}$ is pure and $(\dim F_{k} -1 )$- dimensional for $k = 2, \ldots , t$. Such an ordering is called a \emph{shelling}.  
\end{defn}

Any geometric realization of an abstract simplicial complex is a topological space, and we can often understand the topology of these spaces combinatorially.  As we see in the next section, this is possible for shellable simplicial complexes in way that is convenient for our approach to the new proof we present.  First, we review some more combinatorial properties of simplicial complexes.  One useful combinatorial invariant simply counts the faces of each dimension of a finite simplicial complex $\Delta$:

\begin{defn}\label{f_vector_defn}
The \emph{$f$-vector} $f_{\Delta} = ( f_{0}, f_{1}, \ldots , f_{d})$ is the integer vector with entries $f_{i}$ counting the number of faces of dimension $i$.  The maximal entry $f_{d}$ counts the number of facets of $\Delta$, and $\dim \Delta = d$.  If we consider the empty set to be a face of a simplicial complex $\Delta$, we say $\emptyset$ is a face with dimension equal to $-1$, and $f_{\Delta} = (f_{-1}, f_{0}, f_{1}, \ldots , f_{d})$, where $f_{-1} = 1$. 
\end{defn}

Another combinatorial invariant arises as the alternating sum of the entries in the $f$-vector $f_{\Delta}$:   

\begin{defn}\label{reduced_euler_defn}
The \emph{reduced Euler characteristic} of the simplicial complex $\Delta$ is the alternating sum
$$
\widetilde{\chi}(\Delta) = \sum_{i = -1}^{d} (-1)^{i} f_{i}
$$
where $f_{\Delta} = (f_{-1}, f_{0}, \ldots , f_{d} )$.
\end{defn}

Note that the alternating sum in Definition \ref{reduced_euler_defn} of the $f$-vector where $\emptyset$ is \emph{not} included as a face is simply called the \emph{Euler characteristic}.  So, the modifier \emph{reduced} in this context specifically indicates the inclusion of $\emptyset$ as a face.



 A very useful set of topological invariants of a simplicial complex is the set of Betti numbers.  To define Betti numbers we must first understand the notion of  the (simplicial) homology groups of a simplicial complex.  For an introduction to simplicial homology, see, for example, Section 2.1 \cite{AT}.  For a simplicial complex $\Delta$, let $\Delta_{k}$ denote the set of all $k$-dimensional simplices in $\Delta$, i.e. the set of all simplices in $\Delta$ with $k + 1$ vertices.

A \emph{simplicial $k$-chain} is a formal sum of $k$-simplices $\sum_{i = 1}^{j} c_{i} \sigma_{i}$ where $\sigma_{i} \in \Delta_{k}$ and $c_{i} \in \mathbb{Z}$. Let $C_{k}$ denote the free abelian group with the basis given by the elements of $\Delta_{k}$. The group $C_{k}$ is often called a \emph{chain group}.  Let $\sigma = \{ v_{1} , \ldots , v_{k + 1} \} \in \Delta_{k}$.  The \emph{$k$th boundary map} $d_{k} : C_{k} \to C_{k-1}$ between chain groups is the function defined by 
$$
\partial_{k} (\sigma) = \sum_{m = 1}^{k+1} (-1)^{m} \{ v_{1}, \ldots , \widehat{v_{m}} , \ldots , v_{k + 1} \}
$$ 
where  $\{ v_{1}, \ldots , \widehat{v_{m}} , \ldots , v_{k + 1} \}$ is the $(k-1)$-simplex obtained by omitting the vertex $v_{m}$.  The elements of the subgroup $\ker \partial_{k}$  of $C_{k}$ are called \emph{cycles} and the elements of the subgroup $\im \partial_{k+1}$ of $C_{k}$ are called \emph{boundaries}.   It is simple to verify that $\im \partial_{k + 1} \subset \ker \partial_{k}$, so that the quotient group $H_{k} = \ker \partial_{k} / \im \partial_{k + 1}$ is defined.  We call the group $H_{k}$ the \emph{$k$th homology group} of $\Delta$, and also write $H_{k}(\Delta)$ when the specific simplicial complex under discussion must be made clear. 

\begin{defn}\label{Betti_defn}
The number  $\beta_{k}(\Delta) = \rank(H_{k}(\Delta))$ is the  \emph{$k$th Betti number} of $\Delta$.  
\end{defn}

As we obtain the reduced Euler characteristic by considering the $f$-vector with $\emptyset$ as a face,  we can obtain the reduced Betti numbers, which, abusing notation we will also refer to as $\beta_{k}(\Delta)$, by including $-1$ in their index set.  We will work with the reduced Betti numbers for the remainder of this manuscript.  The most useful way for our purposes to think of the numbers $\beta_{k}(\Delta)$ is as the number of $k$-dimensional holes that $\Delta$ has as a topological space, along with the fact that in the reduced context $\beta_{0}(\Delta)$ is one less than the number of connected components of the space $\Delta$.

\section{Approach to the New Proof}\label{approach_section}

We now explain the approach to the new proof of the identity (\ref{DixID}) presented in this manuscript.  First we must state the next theorem, which was stated for pure simplicial complexes in \cite{bjorner80} and first appears for general, i.e. not necessarily pure, simplicial complexes, in \cite{NonPure1}; this is mentioned because as we will soon see, we study a family of non-pure simplicial complexes. Theorem \ref{wedge_spheres_thm} is a fundamental example of the attractive topological properties that shellable simplicial complexes have.

\begin{thm}\label{wedge_spheres_thm}
A shellable simplicial complex has the homotopy type of a wedge of spheres in varying dimensions.  For each dimension $r$, the number of $r$-spheres is the number of $r$-facets whose entire boundary is contained in the union of earlier facets in the shelling order. 
\end{thm}

By Theorem \ref{wedge_spheres_thm}, when a simplicial complex $\Delta$ is shellable (Definition \ref{shellability_defn}), the Betti numbers $\beta_{i}(\Delta)$ can be interpreted as counting the number of $i$-dimensional faces attaching to $\Delta$ along their entire boundary in a shelling order.   Our approach to reestablishing (\ref{DixID}) is to  present, for each $n$, a shellable simplicial complex $\Delta(n)$ with face numbers $f_{s-1} =  { n \choose s}^{3}$ and then calculate the Betti numbers of $\Delta(n)$.  Fix $d$ as the maximum dimension of the simplicial complex $\Delta$. Then the Euler-Poincar\'{e} relation (attributed \cite{BilleraBjorner} to Henri Poincar\'{e}:)

\begin{equation}\label{EulerPoincareRelation}\sum_{i = -1}^{d}  (-1)^{i} f_{i} = \sum_{i = -1}^{d} (-1)^{i} \beta_{i} \end{equation}

provides a new way of understanding and proving \eqref{DixID}. Our suitable family of simplicial complexes $\{ \Delta(n) | n \geq 1 \}$ was given to us by Patricia Hersh \cite{HershCommunication2013}; this project began as a chapter in the first author's Ph. D. Thesis \cite{DavidsonThesis}.  We define a simplicial complex for each $n$ as follows: 

\begin{defn}[Hersh]\label{the_complex}

Fix $n \geq 1$. Let $\Delta(n)$ be the simplicial complex with  vertices given by  3-tuples $(i_{s} ,j_{s},k_{s})$ for $i_{s} ,j_{s} ,k_{s} \in [n]$ and faces given by  collections of vertices 

$$\{ ({i_1},{j_1},{k_1}), \dots ,({i_r},{j_r},{k_r}) \}$$

 satisfying

$$i_{1}<i_{2}< \dots < i_{r} \quad \text{ and} \quad  j_{1}<j_{2}<\dots < j_{r} \quad \text{ and } \quad
k_{1}<  k_{2} <\dots <k_{r}. $$ 

\end{defn}

The number of $r$-faces of $\Delta(n)$ is counted by the product ${n \choose r + 1}^{3}$.  So
$$
\overline{\chi}(\Delta(n)) = \sum_{s = 0}^{n} (-1)^{s + 1} { n \choose s}^{3}  = (-1) \times \sum_{s = 0}^{n} (-1)^{s} { n \choose s}^{3}.
$$

\subsection{Shellability and Homology Calculations}\label{shellability_homology_section}

We can calculate the Betti numbers  (Definition \ref{Betti_defn}) $\beta_{k}(\Delta)$ of a simplicial complex $\Delta$ by understanding how a shelling order puts $\Delta$ together in a fashion that lets us explicitly understand the topology of $\Delta$.  

\begin{defn}\label{homology_facet_defn}
An $r$-dimensional facet $F_{k}$ of $\Delta$ is  a  \emph{homology $r$-facet} if $F_{k}$ satisfies 
$$
\partial F_{k} = F_{k} \cap \bigcup_{i < k} \overline{F_{i}}.
$$
\end{defn}
So $F_{k}$ is a homology $r$-facet when $F_{k}$ attaches to $\Delta$ along its whole boundary in a shelling order. The Betti numbers of any shellable simplicial complex have a natural interpretation in terms of homology facets:  the number of $r$-spheres in the homotopy type of $\Delta$ is the number of homology $r$-facets, as described in Theorem \ref{wedge_spheres_thm}.  In other words, $\beta_{r}(\Delta)$ is equal to the number of $r$-spheres in the homotopy type of $\Delta$.  


\section{Some Facts About $\Delta(n)$}\label{facts_section}

When $n = 1$, the only nonempty face of $\Delta(n)$ is $\{ (1,1,1) \}$.  Figures \ref{NEquals2_figure} and \ref{NEquals3_figure} show the simplicial complexes  $\Delta(2)$ and $\Delta(3)$, respectively.  It is immediately apparent that $\Delta(n)$ is non-pure for all $n > 1$, and always has precisely one $(n-1)$-dimensional facet given by the collection of vertices $$\{ (1,1,1), (2,2,2), \ldots, (n,n,n) \}.$$ This is the maximum possible dimension of a face of $\Delta(n)$, so the Euler-Poincar\'{e} relation becomes

\begin{equation}\label{EPDN}  \sum_{i = -1}^{n-1}  (-1)^{i} f_{i} = \sum_{i = -1}^{n-1} (-1)^{i} \beta_{i}. \end{equation} 

Note that each of $\Delta(1)$, $\Delta(2)$, and $\Delta(3)$ contain isolated vertices.  Since $\Delta(1)$ is a point, it is pure and connected, but $\Delta(2)$ and $\Delta(3)$ are disconnected and non-pure. 

\begin{figure}[hbtp]
\centering

\includegraphics[width=0.6\textwidth]{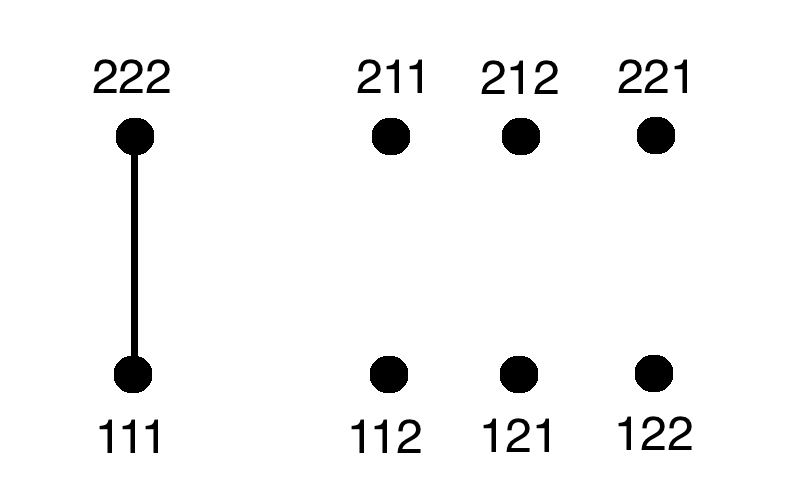}

\caption{The simplicial complex $\Delta(2)$.}
\label{NEquals2_figure}
\end{figure}

\begin{figure}[hbtp]
\centering

\includegraphics[width=0.85\textwidth]{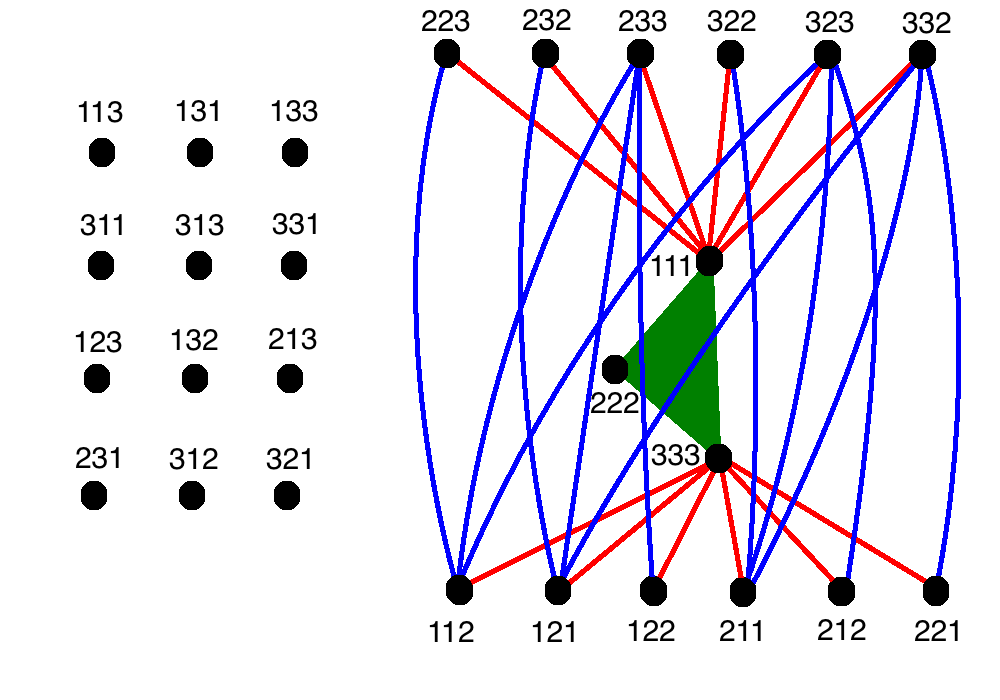}

\caption{The simplicial complex $\Delta(3)$.}
\label{NEquals3_figure}
\end{figure}

Before establishing a shelling order for $\Delta(n)$, we must understand which faces are facets:

\begin{lemma}\label{facets_lemma}
Let $F = \{v_{1}, \ldots , v_{r}\} = \{ (i_{1}, j_{1}, k_{1}), \ldots , (i_{r}, j_{r}, k_{r}) \}$ be a face of $\Delta(n)$.  Then $F$ is a facet if and only if $F$ satisfies the following three properties:
\begin{itemize}
\item[(P1)]  $\max \{ i_{r}, j_{r}, k_{r} \} = n$.
\item[(P2)] $\min\{i_{1}, j_{1}, k_{1} \} = 1$.
\item[(P3)] If $r \geq 2$, $\min \{i_{\ell + 1} - i_{\ell}, j_{\ell + 1} - j_{\ell}, k_{\ell + 1} - k_{\ell} \} = 1$ for all $\ell \in [r-1]$. 
\end{itemize}
\end{lemma}

\begin{proof}
Let $F =  \{v_{1}, \ldots , v_{r}\}$ be a facet of $\Delta(n)$.  Properties P1 and P2 clearly must hold for $F$: If P1 does not hold $F \subset F \cup \{v_{r + 1} \}$, where $v_{r + 1} = ( i_{r} + 1, j_{r} + 1, k_{r} + 1 )$-or one could just add $(n,n,n)$ as a vertex to obtain a facet. If P2 does not hold, then $F \subset \{ v_{0} \} \cup F$ where $v_{0} = (i_{1}-1, j_{1}-1, k_{1}-1)$ or one could simply add $(1,1,1)$ as $v_{0}$ to obtain a facet.  If P3 does not hold, there exists an index $\ell \in [r-1]$ where $\min \{i_{\ell + 1} - i_{\ell}, j_{\ell + 1} - j_{\ell}, k_{\ell + 1} - k_{\ell} \} > 1$ and we can construct a vertex $v_{s} = (i_{s}, j_{s}, k_{s} )$ satisfying

$$ 
i_{\ell} < i_{s} < i_{(\ell+1)},   \quad j_{\ell} < j_{s} < j_{(\ell+1)},   \quad \text{and}  \quad 
k_{\ell} < k_{s} < k_{(\ell+1)}.  
$$
Then $F ' = \{v_{1}, v_{2}, \ldots ,  v_{\ell}, v_{s}, v_{\ell + 1}, \ldots , v_{r}\}$ properly contains $F$ and $F$ cannot be a facet.  So each condition is necessary and sufficient for $F$ to be a facet.  Informally, a vertex ``filling in the index difference in one of the positions" could result in a facet. \end{proof}

From this characterization of the facets of $\Delta(n)$, we immediately obtain the next lemma:

\begin{lemma}\label{disconnected_lemma}
For $n \geq 2$, the simplicial complex $\Delta(n)$ is non-pure and disconnected.
\end{lemma}

\begin{proof}
Any vertex $v_{s} = (i_{s}, j_{s}, k_{s})$ satisfying $\{ 1, n \} \subset \{ i_{s}, j_{s}, k_{s} \}$ must be an isolated vertex, as $v_{s}$ cannot be contained in any other face in this case.  If $n > 1$ there is more than one vertex in $\Delta(n)$, so $\Delta(n)$ contains isolated vertices and is disconnected for all $n \geq 2$.  For all $n$ there is an $(n-1)$-dimensional facet $F(n) = \{ (1,1,1), (2,2,2), \ldots , (n,n,n)\}$, and for $n \geq 2$, $\dim F(n) > 0$.   So $\Delta(n)$ is not pure for $n \geq 2$.   \end{proof}

It is also useful to obtain new faces of $\Delta(n)$ from old, and understand how to obtain a new facet from an old facet.  To make these actions possible, we now define operations on the faces of $\Delta(n)$ in Definitions \ref{up_twist}, \ref{down_twist}, and \ref{safe_twist}. 

\begin{defn}\label{up_twist}
Let $F = \{v_{1}, v_{2}, \ldots , v_{r} \}$ be a face of $\Delta(n)$.  For $\ell \in [r-1]$, let  
$$
 A_{\ell} = \{ a_\ell \in \{i_{\ell}, j_{\ell} ,k_{\ell} \}  : a_{\ell + 1} - a_{\ell} > 1\}
$$ 
and let 
$$
 A_{r} = \{ a_r \in \{i_{r}, j_{r} ,k_{r} \}  :  a_{r} < n \}.
$$
 For $\ell$ such that $A_{\ell} \neq \emptyset$,  define a \textbf{up-twist} $G$ about the vertex $v_{\ell}$ as the face of $\Delta(n)$ obtained by replacing $v_{\ell}$ in $F$  with the new vertex $v_{\ell}'$ of $\Delta(n)$ obtained by  increasing an index of $v_{\ell}$ in $A_{\ell}$ by 1.  
\end{defn}

\begin{defn}\label{down_twist}
Let $F = \{v_{1}, \ldots , v_{r}\}$ be a face of $\Delta(n)$.  For $\ell \in \{2, \ldots , r\} $, let 
$$
B_{\ell} = \{ b_{\ell} \in \{i_{\ell}, j_{\ell} ,k_{\ell} \} :  b_{\ell} - b_{\ell-1}  > 1\}
$$  
and let
$$
B_{1} = \{ b_{1} \in \{i_{1}, j_{1} ,k_{1} \} : b_{1}   > 1\}
$$  

For $\ell$ such that $B_{\ell} \neq \emptyset$,  define a \textbf{down-twist} $G$  about the vertex $v_{\ell}$ as the face  of $\Delta(n)$ obtained by replacing $v_{\ell}$ in $F$  with the new vertex  $v_{\ell}'$ of $\Delta(n)$ obtained by  decreasing an index of $v_{\ell}$  in $B_{\ell}$ by 1.  
\end{defn}

\begin{ex} In this example, we consider faces of $\Delta(5)$. In the face $F = \{ (1,1,2), (2,5,5) \}$ with $v_{1} = (1,1,2)$, and $v_{2} = (2,5,5)$ the set $A_{1} = \{ j_{1}, k_{1} \} = \{1, 2\}$, and $B_{2} = \{ j_{2}, k_{2} \} = \{5,5\}$.  An up-twist of $F$ about $v_{1}$ is the new face $\{ (1,2,2),(2,5,5) \}$.  A down-twist of $F$ about $v_{2}$ is the new face $\{(1,2,2),(2,4,5)\}$.

\end{ex}

\begin{defn}\label{safe_twist}

Recall conditions P1, P2, and P3 from Lemma \ref{facets_lemma}, and let $F$ be a facet.  We say an up-twist $G$ of a facet $F = \{v_{1}, v_{2}, \ldots , v_{r} \}$ is \textbf{safe} if either $\ell > 1$ and condition P3 is conserved, or $\ell = 1$ and condition P2 is conserved.  We say a down-twist $G$ of a facet $F$ is safe if either $\ell < r$ and condition P3 is conserved, or $\ell = r$ and condition P1 is conserved.
\end{defn}

\begin{ex}
In this example, we again consider faces of $\Delta(5)$.  The down-twist $$\{(1,2,2), (2,4,4)\}$$ of $$\{ (1,2,2),(2,4,5) \}$$ about the vertex $(2,4,5)$ is not safe, because condition P1 is not conserved.   The down-twist $$\{ (1,2,2), (2,3,5)\}$$ of $$\{ (1,2,2),(2,4,5) \}$$ about the vertex $(2,4,5)$ is safe, because P1 is conserved. 
\end{ex}

\begin{lemma}\label{twisting_lemma}
Let $F_{1} = \{v_{1}, \ldots , v_{r} \}$ be a facet of $\Delta(n)$.  If $F_{2}$ is a safe up-twist or a safe down-twist of $F_{1}$,  then $F_{2}$ is a facet of $\Delta(n)$.
\end{lemma}

\begin{proof}
First we consider the case where  $F_{1}$ is a facet $\{v_{1}, \ldots , v_{r} \}$ of $\Delta(n)$ and
$$
F_{2} = \{ v_{\ell'} \} \cup \{ v_{j} : j \in [r] \setminus \{ \ell\} \}
$$
is a safe up-twist of $F_{1}$ about the vertex $v_{\ell}$.  By Lemma \ref{facets_lemma}, it is sufficient to show that $F_{2}$ satisfies P1, P2, and P3. If $\ell < r$, then $v_{r}$ is unaffected by the vertex change and P1 holds for $F_{2}$.  If $\ell = r$, then the maximum element of $v_{r}'$ will not decrease and P1 is satisfied by $F_{2}$. If $\ell > 1$, then P2 is trivially satisfied by $F_{2}$.  If $\ell =1$,  then since $F_{2}$ is a \textit{safe} up-twist, $\min\{i_{1},j_{1},k_{1} \} = \min\{i_{1}', j_{1}',k_{1}'\} = 1$, where $(i_{1}', j_{1}', k_{1}')$ is the new vertex in $v_{1}'  \in F_{2}$.  So $P2$ holds.

Now, since $F_{1}$ is a facet, $\min \{i_{\ell + 1} - i_{\ell}, j_{\ell + 1} - j_{\ell}, k_{\ell + 1} - k_{\ell} \} = 1$.  Without loss of generality, we can say $\min \{i_{\ell + 1} - i_{\ell}, j_{\ell + 1} - j_{\ell}, k_{\ell + 1} - k_{\ell} \}  = i_{\ell + 1} - i_{\ell}$, so that $i_{\ell} \notin A_{\ell}$ and $v_{\ell}' = (i_{\ell}, j_{\ell}', k_{\ell}')$.  Therefore 
$\min \{i_{\ell + 1} - i_{\ell}, j_{\ell + 1} - j_{\ell}', k_{\ell + 1} - k_{\ell}' \} = 1$ and P3 holds.

A similar argument shows that if $F_{2}$ is a safe down-twist of $F_{1}$, then $F_{2}$ is also a facet of $\Delta(n)$.  \end{proof}

\section{A Shelling Order for $\Delta(n)$}\label{shelling_order_section}

In this section we construct a shelling order for $\Delta(n)$. Recall that by Lemma \ref{disconnected_lemma}, $\Delta(n)$ is not pure.   To set up the shelling order, we first partition $\Delta(n)$ into sets of facets according to dimension.  For $0 \leq m \leq n-1$, let $S_{m}$ be the set of facets of dimension $m$.  For example, $S_{n-1}$ is the set containing the single $(n-1)$-dimensional facet $F(n)=  \{(1,1,1), (2,2,2), (3,3,3), ... , (n,n,n)  \}$, and  $S_{0}$ contains the facets comprised solely of isolated vertices such as the facet $F = \{ (1,n,1) \}$.  It is intuitive that such our shelling order would require respecting the dimension of facets of $\Delta(n)$ due to elementary facts about non-pure shellable complexes (see Lemma 2.2 in \cite{NonPure1}).

Non-pure shellable simplicial complexes were not studied in detail before the work in \cite{NonPure1}, and we highlight the fact that our main object is not a  simplicial complex in order to (1) build on this body of research by providing additional motivating examples and (2) providing a fundamental cross-disciplinary (within the partitioned field of combinatorics) application of an existing, fully-developed toolkit. The next definition allows us to order the facets in $S_{m}$ for a fixed $m$ using the lexicographic order.  

\begin{defn}\label{sigma_defn} The \textit{$\sigma$-word} $\sigma(F)$ of the face $F = \{v_{1}, \ldots ,v_{r} \}$ is the sequence $$( i_{1}, j_{1}, k_{1}, i_{2}, j_{2}, k_{2}, i_{3}, \ldots , k_{r-1}, i_{r}, j_{r}, k_{r} ). $$  \end{defn}

Informally, we can see that the $\sigma$-word of a face is obtained by simply ignoring all the parentheses in the listing of the vertices of the face. 

\begin{ex}
The $\sigma$-word of the vertices of the facet $$F = \{(1,2,1), (3,3,3), (4,4,5) \} $$ of $\Delta(5)$ is $$\sigma(F) = (1,2,1,3,3,3,4,4,5 ). $$ The $\sigma$-word of the vertices of the facet $$G = \{ (1,2,2), (2,3,3), (4,4,5) \}$$ of $\Delta(5)$ is $$\sigma(G) = (1,2,2,2,3,3,4,4,5 ). $$ In the lexicographic order we have $\sigma(F) < \sigma (G)$.  
\end{ex}

Now we define the order on the facets of $\Delta(n)$ that we will show is a shelling order. 

\begin{defn}\label{O_defn} Define an order $\mathcal{O}$ on the facets of $\Delta(n)$ as follows: $F_{i} < F_{j}$ in the  order $\mathcal{O}$ if
\begin{enumerate}
\item $F_{i} \in S_{m}$ and $F_{j} \in S_{\ell}$ for $\ell< m$ or
\item $\ell = m$ and $\sigma(F_{i})$ is lexicographically smaller than  $\sigma(F_{j})$.  
\end{enumerate}
\end{defn}

\begin{thm}\label{shelling_thm}
The order $\mathcal{O}$ is a shelling order for the facets of $\Delta(n)$. 
\end{thm}

The remainder of this section will be devoted to showing that $\mathcal{O}$ is a shelling order. We need the following well-known lemma, which provides a useful working definition of a shelling, in our proof of Theorem \ref{shelling_thm}.  This lemma is explicitly stated for non-pure simplicial complexes as Lemma 2.3 in \cite{NonPure1}, but as we have see for other facts about non simplicial complexes, a earlier version for pure simplicial complexes appears in \cite{bjorner80}. 

\begin{lemma}\label{shelling_lemma}
An order $F_{1}, F_{2}, \ldots , F_{t}$ of the facets of a simplicial complex $\Delta$ is a shelling if and only if for every $i$ and $k$ satisfying $1 \leq i < k \leq t$ there is a $j$ with $1 \leq j < k $ and a vertex $v \in F_{k}$ such that $F_{i} \cap F_{k} \subset F_{j} \cap F_{k} = F_{k} \setminus \{ v \}$.  
\end{lemma}

We will follow the notation of Lemma \ref{shelling_lemma} and let $[t]$ denote the index set for the order $\mathcal{O}$.  To work with Lemma \ref{shelling_lemma} in the proof of Theorem \ref{shelling_thm}, we will be fixing two facets $F_{i}$ and $F_{k}$ and constructing a facet $F_{j}$ satisfying the conditions of the lemma.  To make this easier, we now develop some notation for vertex subsets of $F_{i}$ and $F_{k}$.  

 Let $F_{k}$, for $k > 1$, be a facet of $\Delta(n)$.   Let $i$ be an index satisfying $1 \leq i < k \leq t$.  Let $\mathcal{V}_{i,k}$ denote the (possibly empty) set of vertices in $F_{i} \cap F_{k}$, let $\mathcal{V}_{k} = F_{k} \setminus F_{i}$, and let $\mathcal{V}_{i} = F_{i} \setminus F_{k}$.   Write $\mathcal{V}_{i,k} = \{ v_{c,1}, \ldots , v_{c,s} \}$, $\mathcal{V}_{k} = \{ v_{k,1}, \ldots , v_{k,e} \}$, and $\mathcal{V}_{i} = \{v_{i,1}, \ldots , v_{i,u} \}$. We write the vertex sets so that as positive integers, $(c,1) < \cdots < (c,s)$, $(k,1) <  \cdots < (k,e)$, and $(i,1) <  \cdots < (i,u)$.  Also, we write the indices of $\mathcal{V}_{i,k}$, $\mathcal{V}_{k}$, and $\mathcal{V}_{i}$ in the same order they appear in $F_{k}$ and $F_{i}$, and we do not rename the indices when considering the subsets $\mathcal{V}_{i,k} , \mathcal{V}_{i} $, and $\mathcal{V}_{k}$.

\begin{ex} Let $$F_{i} = \{ (1,2,1), (3,3,3), (4,4,4), (5,5,5) \}$$ and 

$$F_{k} = \{ (1,2,1), (2,3,3), (5,4,5) \}.$$   

Then $\mathcal{V}_{i,k} = \{ (1,2,1)\}$, $$\mathcal{V}_{k} = (2,3,3), (5,4,5) \}$$, and $$\mathcal{V}_{i} = \{(3,3,3), (4,4,4), (5,5,5) \}.$$  Also, $\{ (c,1) \} = \{ 1 \}$, $\{ (k,1), (k,2) \} = \{ 2, 3 \}$ and $\{  (i,1), (i,2), (i,3) \} = \{ 2,3,4 \}$.  

\end{ex}

\noindent The next lemma will make it easier to work with the sets $\mathcal{V}_{i,k}$, $\mathcal{V}_{i}$, and $\mathcal{V}_{k}$. 
 
\begin{lemma}\label{vertex_partition_lemma}
Let $F_{i}$ and $F_{k}$ be facets of $\triangle(n)$ such that $i < k \in [t]$.  There exist ordered partitions of $\mathcal{V}_{i,k}, \mathcal{V}_{i}$, and $\mathcal{V}_{k}$ into blocks of ordered vertices
$$
\mathcal{V}_{i,k} = C_{1} | \cdots | C_{N},
$$

$$
\mathcal{V}_{i} = I_{1} | \cdots | I_{M},
$$
and 
$$
\mathcal{V}_{k} = K_{1} | \cdots | K_{M}
$$
such that each ordered block of the ordered partitions corresponds to a consecutive subsequence of vertices in a facet. 

\end{lemma}

\begin{proof}
We can generate the required partitions of the vertices of $\mathcal{V}_{i}$, $\mathcal{V}_{k}$, and $\mathcal{V}_{i,k}$ using an algorithmic approach.  We will explain the algorithm $\mathcal{V}_{i,k} = C_{1} | \cdots | C_{N}$; the algorithms for $\mathcal{V}_{i}$ and $\mathcal{V}_{k}$ are similar.  

Let $F_{i} = \{ v_{1}, \ldots , v_{r} \}$.  If  $\mathcal{V}_{i,k}  = \emptyset$, then the partition is empty, and there is nothing to compute.  So, assume $\mathcal{V}_{i,k} \neq \emptyset$.  We use the following algorithm to build the blocks of the ordered partition $C_{1} | \cdots | C_{N}$.
\begin{alg} 

\begin{itemize}

\item Input:  The vertices $\{ v_{1}, \ldots , v_{r} \}$ of $F_{i}$ and the vertex subset $\mathcal{V}_{i,k}$.

\item Output: An ordered partition of  $\mathcal{V}_{i,k}$ of ordered blocks of vertices in $\mathcal{V}_{i,k}$, in which each block is a set of vertices that are both consecutive in $\{ v_{1}, \ldots , v_{r} \}$ and written in the order that they appear in $\{ v_{1}, \ldots , v_{r} \}$.

\item Initialize $\ell(1) = \min\{ \ell \in [r] : v_{\ell} \in \mathcal{V}_{i,k} \}$, set $C_{1} =  \{ v_{\ell(1)} \}$ .

\item While $\ell(i) < r$: 
\begin{itemize}

\item If $v_{\ell(i) + 1} \in \mathcal{V}_{i,k}$, set $C_{i} = C_{i} \cup \{ v_{\ell(i) + 1} \}$, update $\ell(i) = \ell(i) + 1$. 
\item Else if $v_{\ell(i) + 1} \notin \mathcal{V}_{i,k}$, set $v_{\ell(i)}$ as the last vertex in $C_{i}$. 

\begin{itemize}

\item If the set $\{ m \in [r] : m > \ell(i) \ \text{and} \ v_{m} \in \mathcal{V}_{i,k} \}$ is empty, $C_{i} = C_{N}$ and the algorithm terminates.  
\item Else update $\ell(i + 1)  = \min \{ m \in [r] :  m > \ell(i) \ \text{and} \ v_{m} \in \mathcal{V}_{i,k} \}$ and set $C_{i + 1} = \{ v_{\ell(i + 1)} \}$.

\end{itemize}

\end{itemize}

\item Return the partition $\mathcal{V}_{i,k} = C_{1} | \cdots | C_{N}$.  

\end{itemize}

\end{alg}


We explain why the partitions of $\mathcal{V}_{i}$ and $\mathcal{V}_{k}$ both have the same number of blocks $M$: assume by way of contradiction that the partition of $F_{i}$ has more blocks than the partition of $F_{k}$.  Then either $(i)$ there is at least one vertex in the sequence $\{v_{1} , \ldots , v_{r} \}$ that is between two vertices of $F_{k}$, $(ii)$ there is a vertex of $F_{i}$ greater than the last vertex of $F_{k}$, or $(iii)$ there is a vertex smaller than the first vertex of $F_{k}$.  Case $(i)$ implies $F_{k}$ does not satisfy P3, Case $(ii)$ implies that $F_{k}$ does not satisfy P1, and Case $(iii)$  implies $F_{k}$ does not satisfy P2.  So, all three cases are impossible by Lemma \ref{facets_lemma}.  So $F_{i}$ cannot have more blocks in the ordered partition than $F_{k}$.  The argument is symmetric in $F_{i}$ and $F_{k}$, so the ordered partitions of $F_{i}$ and $F_{k}$ have the same number of blocks.  \end{proof}

\begin{ex}
Here is an example of the ordered partitions with ordered blocks described in Lemma \ref{vertex_partition_lemma}. Define the facet  $F_{i}$ as  $$\{ (1,2,1), (2,3,3), (3,4,4), (5,5,5), (6,6,6), (7,8,8), (9,9,9), (10,10,10) \}$$ and let $$F_{k} = \{ (1,2,1), (2,3,5), (6,6,6), (7,8,9), (10,10,10) \}.$$ Note that $F_{i}$ and $F_{k}$ are both facets of $\Delta(10)$.   Then $$\mathcal{V}_{i,k} =  \{ (1,2,1), (6,6,6) ,(10,10,10) \},$$ $$\mathcal{V}_{i} = \{ (2,3,3), (3,4,4), (5,5,5), (7,8,8), (9,9,9) \},$$ and $$\mathcal{V}_{k} = \{ (2,3,5), (7,8,9) \}.$$ We have $$\mathcal{V}_{i,k} = C_{1} | C_{2} | C_{3} =  \{ (1,2,1) \} | \{(6,6,6) | \{ (10, 10, 10) \},$$ $$\mathcal{V}_{i} = I_{1} |  I_{2} = \{ (2,3,3), (3,4,4), (5,5,5), \} | \{  (7,8,8), (9,9,9) \},$$ and $$\mathcal{V}_{k} = K_{1} |  K_{2} = \{ (2,3,5) \} | \{ (7,8,9)  \}. $$  
\end{ex}

The next lemma is useful in our proof of Theorem \ref{shelling_thm}.  Recall that we write $\mathcal{V}_{k} = \{ v_{k,1}, \ldots , v_{k,e} \}$.

\begin{lemma}\label{non_empty_B}
Let $i < k$ be indices in the order $\mathcal{O}$.  There exists $\ell \in \{ (k,1) \ldots , (k,e) \}$ such that $B_{\ell} \neq \emptyset$.
\end{lemma}

\begin{proof}
First we handle the case where $\dim F_{i} = \dim F_{k}$. In this case $F_{i}$ and $F_{k}$ each have $r$ vertices, and the sequences  $\sigma(F_{i})$ and $\sigma(F_{k})$ are both of length $3r$.   By our construction of the shelling order $\mathcal{O}$ in Definition \ref{O_defn} this implies $\sigma(F_{i}) < \sigma(F_{k})$ in the lexicographic order which means the first place the two sequences differ, call this index $b \in [3r]$, is larger in $\sigma(F_{k})$.  

In other words, we have
$$
\sigma(F_{i}) = ( p_{1}, \ldots , p_{3r} )
$$
and
$$
 \sigma(F_{k}) = ( q_{1}, \ldots , q_{3r} )
 $$
 where $p_{a} = q_{a}$ for $a \in [3r]$ satisfying $a < b$, and $q_{b} > p_{b}$ as integers.  
 
The first place the sequences differ occurs in the vertex of smallest index not present in both $F_{i}$ and $F_{k}$.  So $q_{b} \in  \{i_{k,1}, j_{k,1}, k_{k,1} \}$ as $v_{k,1}$ is the vertex of smallest index in $\mathcal{V}_{k}$.  Without loss of generality, we can say $b$ designates the position of $i_{k,1}$.  Recall that we write $\mathcal{V}_{i} = \{v_{i,1}, \ldots , v_{i,u} \}$.   Then we have $ i_{k,1} > i_{i,1}$. If $(k,1) = (i,1) = 1$, then $i_{k,1} \geq 2$ and $B_{k,1} \neq \emptyset$.  If $(k,1) > 1$, $i_{(k,1) - 1}$ must appear in a vertex in $\mathcal{V}_{i,k}$ by our choice of $b$, and we can write $i_{k,1} - i_{(k,1)-1} > i_{i,1}- i_{(k,1)-1} \geq 1$.  So $B_{k,1} \neq \emptyset$ in this case.  


Next we handle the case where $\dim F_{i} > \dim F_{k}$.  We can write $F_{i}$ and $F_{k}$ as the disjoint unions
$$
F_{k} = \mathcal{V}_{i,k} \sqcup \mathcal{V}_{k}, \quad F_{i} = \mathcal{V}_{i,k} \sqcup \mathcal{V}_{i}.
$$
We know that $|\mathcal{V}_{i}| > |\mathcal{V}_{k}|$ because $\dim F_{i} > \dim F_{k}$.  By Lemma \ref{vertex_partition_lemma} there exist partitions of  $\mathcal{V}_{i,k}, \mathcal{V}_{i}$, and $\mathcal{V}_{k}$ 

$$
\mathcal{V}_{i,k} = C_{1} | \cdots | C_{N}, \quad  \mathcal{V}_{i} = I_{1} | \cdots | I_{M},
$$
and 
$$
\mathcal{V}_{k} = K_{1} | \cdots | K_{M}
$$
such that each block in each partition corresponds to an uninterrupted sequence of vertices in a facet.  Since $|\mathcal{V}_{i}| > |\mathcal{V}_{k}|$ and the ordered partitions of $\mathcal{V}_{i}$ and $\mathcal{V}_{k}$ have the same number of blocks, there must exist $A \in [M]$ such that $|I_{A}|  > |K_{A}|$.     For indices $x \in [u]$ and $y \in [e]$ we can write
$$
I_{A} = \{v_{i ,x}, \ldots , v_{i, (x+ |I_{A}|)} \}, \quad \text{and} \quad K_{A} = \{v_{k,y}, \ldots , v_{k,(y + |K_{A}|)} \}.
$$
Recall that the $\sigma$-word $\sigma(K_{A})$ (Definition \ref{sigma_defn}) is the ordered set of indices of all vertices appearing in the face $K_{A}$.  We divide the proof for $\dim F_{i} > \dim F_{k}$ into two  cases \begin{equation}\label{C1} n \in  \sigma(K_{A}),\end{equation}  and \begin{equation}\label{C2} n \notin  \sigma(K_{A}) .\end{equation}

Consider first the case (\ref{C1}). This case implies $$n \in  \{ i_{k,(y + |K_{M}|)}, j_{k, (y + |K_{M}|)}, k_{k, (y + |K_{M}|)} \},$$ the index set of the last vertex in the last block of the partition  $K_{1} | \cdots | K_{M}$ of $\mathcal{V}_{k}$.  We can assume without loss of generality that $n = i_{k,(y+ |K_{M}|)}$.  Then $n$ appears as an element of $v_{i, (x + |I_{M}|)} $ also.  Because of this, we know that in $F_{i}$, the vertices in $I_{M}$ immediately follow the vertices in $C_{N}$, and in $F_{k}$, the vertices in $K _{M}$ immediately follow the vertices in $C_{N}$.   For some $z \in [s]$, we can write $C_{N} = \{ v_{c,z}, \ldots , v_{c, (z + |C_{N}|)} \}$.    

Then $n - i_{c, (z + |C_{N}|)} \geq |I_{M}|$, and the net change in the $i$ index in the vertices of $K_{M}$ is bounded below by $|I_{M}| >  |K_{M}|$, and there are only $|K_{M}|$ vertices to accomplish this change. Therefore there must exist an index $\ell \in \{ (k,y), \ldots ,( k, y + |K_{M}|) \}$ such that $i_{\ell} - i_{\ell -1} > 1$.  So, for this $v_{\ell} \in \mathcal{V}_{k}$, $B_{\ell} \neq \emptyset$.  

Now we consider the case (\ref{C2}).  This implies that for some $w \in [s]$  there exists a vertex $v_{c,w} \in \mathcal{V}_{i,k}$ where $(c,w) =  (k,(y + |K_{A}|)) + 1$ in the vertex numbering in $F_{k}$.  Since $F_{k}$ is a facet, it satisfies P3 from Lemma \ref{facets_lemma}, which means that $\min \{i_{c,w} - i_{k, (y+ |K_{A}|)}, j_{c ,w} - j_{k, (y+ |K_{A}|)}, k_{c ,w} - k_{k, (y+ |K_{A}|)} \} = 1$.  Without loss of generality we can say $i_{c,w} - i_{k, (y+ |K_{A}|)} = 1$.  If $1 \in  \sigma(K_{A})$ then $A = 1$ and $i_{k, (y+ |K_{1}|)} \geq |I_{1}|$ where $|I_{1}| > |K_{1}|$, but we only have $|K_{1}|$ vertices to accomplish this index change and so there exists $\ell \in \{ 1, \ldots , (1 + |K_{1}|) \}$ such that $B_{\ell} \neq \emptyset$. 

If $1 \notin  \sigma(K_{A}) $, there exists $x \in [s]$ and $v_{c, x} \in \mathcal{V}_{i,k}$ such that $(c,x) + 1 = (k,y) $ in the label sequence of the vertices of $F_{k}$.  Then $i_{k, (y + |K_{A}|) } - i_{c, x} \geq |I_{A}|$ where $|I_{A}| > |K_{A}|$.  But we only have $|K_{A}|$ vertices to accomplish this index change and so there exists $\ell \in \{ (k, y), \ldots , (k, (y + |K_{A}|))  \}$ such that $B_{\ell} \neq \emptyset$.  This completes the proof of the Lemma for the case $\dim F_{i} > \dim F_{k}$.  \end{proof}

Now we prove Theorem \ref{shelling_thm}.  The essence of the proof is that given any pair of facets  $F_{i}$ and $F_{k}$ such that $i < k$ in $\mathcal{O}$, we may use Lemma \ref{non_empty_B} to construct a facet $F_{j}$ such that  the hypotheses of Lemma \ref{shelling_lemma} is satisfied, which will show that $\mathcal{O}$ is a shelling order.  

\begin{proof}
Let $F_{i}$ and $F_{k}$ be such that $i$ and $k$ satisfy $1 \leq i < k \leq t$ in the order $\mathcal{O}$.  Recall we write $\mathcal{V}_{i,k} = F_{i} \cap F_{k}$ and $\mathcal{V}_{k} = \{ v_{k, 1}, \ldots , v_{k, e} \}$,  where $\mathcal{V}_{k}  = F_{k} \setminus F_{i}$.  Write $F_{k} = \{ v_{1}, \ldots , v_{r} \}$.  We will find a vertex $v \in F_{k}$ and construct a facet $F_{j}$ such that $1 \leq j < k $ and such that $\mathcal{V}_{i,k} \subset F_{j} \cap F_{k} = F_{k} \setminus \{ v \}$.  This will show that $\mathcal{O}$ is a shelling order by Lemma \ref{shelling_lemma}.  

By Lemma \ref{non_empty_B}  there exists  $\ell \in \{ (k,1), \ldots , (k, e) \}$ such that $B_{\ell} \neq \emptyset$.   We will divide the proof into two cases:  $\ell = r$ and $\ell < r$.   For now assume that $\ell < r$.   If such an $\ell$ exists we choose $\ell$ that is minimal.   

Then choose the ``left-most" vertex element in $B_{\ell}$: for example if $B_{\ell} =  \{ i_{\ell}, k_{\ell} \}$ we choose $i_{\ell}$. Without loss of generality we can say that $i_{\ell}$ is the left-most element of the set $B_{\ell}$.  Let $w = (i_{\ell} -1, j_{\ell}, k_{\ell} )$.   Since $F_{k}$ is a facet, we know that $\min \{ i_{\ell + 1} - i_{\ell},  j_{\ell + 1} - j_{\ell},  k_{\ell + 1} - k_{\ell} \} = 1$.  We now have two sub-cases to consider:  (i):  $\min \{  j_{\ell + 1} - j_{\ell},  k_{\ell + 1} - k_{\ell} \} = 1$ and (ii): $\min \{  j_{\ell + 1} - j_{\ell},  k_{\ell + 1} - k_{\ell} \} > 1$.  In the case (i), the down-twist (Definition \ref{down_twist}) about $v_{\ell}$ 
  $$
  F_{j} = F_{k} \setminus \{v_{\ell} \} \cup \{w\}
  $$
  is safe and $F_{j}$ is a facet.  
  
 In this instance $\dim F_{j} = \dim F_{k}$.  The only place $\sigma(F_{j})$ and $\sigma(F_{k})$ differ is the position of $i_{\ell}-1 $ from the new vertex $w$.  So $\sigma(F_{j}) < \sigma(F_{k})$ and we know $j < k$ in the order $\mathcal{O}$.  Also, $ \mathcal{V}_{i,k} \subset F_{j} \cap F_{k} = F_{k} \setminus \{v_{\ell} \}$, so $ v_{\ell}$ and $F_{j}$ satisfy the conditions of Lemma \ref{shelling_lemma}.

 Next, consider the sub-case (ii): $\min \{  j_{\ell + 1} - j_{\ell},  k_{\ell + 1} - k_{\ell} \} > 1$.  Since $F_{k}$ is a facet, P3 is satisfied and $i_{\ell + 1} - i_{\ell} = 1$ must hold.  In this case the face
 $$
 F_{j} = F_{k} \setminus \{ v_{\ell} \} \cup \{ w , (i_{\ell}, j_{\ell} + 1, k_{\ell}+ 1) \}
 $$
 satisfies P3 and is a facet.  Since $\dim F_{j} > \dim F_{k}$, $j < k$ in $\mathcal{O}$.  Clearly $\mathcal{V}_{i,k} \subset F_{j} \cap F_{k} = F_{k} \setminus \{ v_{\ell} \}$.  
 
Next we consider the case where the only $\ell \in \{(k,1), \ldots, (k,e) \}$ satisfying $B_{\ell} \neq \emptyset$ is $\ell = r$.  Again without  loss of generality we can say that $i_{r}$ is the left-most index in $B_{r}$.  Let $w = (i_{r} -1, j_{r}, k_{r} )$.   There are two sub-cases to consider: (i) $n \in \{j_{r}, k_{r} \}$ and (ii) $n \notin \{j_{r}, k_{r} \}$.
 
 If (i) $n \in \{j_{r}, k_{r} \}$, then the down-twist about $v_{r}$
 $$
 F_{j} = F_{k} \setminus \{ v_{r} \}  \cup \{ w\}
 $$
 is safe and $F_{j}$ is a facet of the same dimension as $F_{k}$ satisfying $\sigma(F_{j}) < \sigma(F_{k})$ and so $j < k$ in $\mathcal{O}$.  For the sub-case (ii) when  $n \notin \{j_{r}, k_{r} \}$, let 
 $$
 F_{j} = F_{k} \setminus \{ v_{r} \} \cup \{ w, (n, n, n) \}.
 $$
 Since $\dim F_{j} > \dim F_{k}$, we have $j < k $ in $\mathcal{O}$.  In both sub-cases $\mathcal{V}_{i,k} \subset F_{j} \cap F_{k} = F_{k} \setminus \{ v_{r} \}$.  This completes the proof.   \end{proof}

\section{The Homology Facets of $\Delta(n)$}\label{homology_section}

Recall that our approach to our new proof (Section \ref{approach_section}) is to calculate the Betti numbers (Definition \ref{Betti_defn}) of $\Delta(n)$ using shelling order  $\mathcal{O}$ in which puts $\Delta(n)$ together as a topological space. The next lemma characterizes the homology facets (Definition \ref{homology_facet_defn}) of $\Delta(n)$ for dimension 1 and greater:

\begin{lemma}\label{homology_facets_lemma}
Let $r \geq 2$.  A facet $F_{k} = \{ v_{1} , \ldots , v_{r } \}$ is a homology $(r-1)$-facet of $\Delta(n)$ if and only if $B_{\ell} \neq \emptyset$ for all $\ell \in [r ]$.  
\end{lemma}

\begin{proof}
First let $B_{\ell} \neq \emptyset$ for all $\ell \in [r]$.  It suffices to show that for all $\ell$,  $F_{k} \setminus \{ v_{\ell} \} \subset F_{j(\ell)}$ for some $j(\ell) < k$.  First, let $\ell = r$.  If at least two of the elements of the set  $  \{ i_{r}, j_{r}, k_{r}\}$ are equal to $n$, then since $B_{r} \neq \emptyset$, we can say without loss of generality that $i_{r} - i_{r-1} > 1$.  Then  the facet 
$$
F_{j(\ell)} = F_{k} \setminus \{v_{r} \} \cup \{ (i_{r} - 1, j_{r}, k_{r}) \}
$$
satisfies $\dim F_{j(\ell)} = \dim F_{k}$ and $\sigma(F_{j(\ell)}) < \sigma(F_{k})$, so $j(\ell) < k$ and we have the desired containment $F_{k} \setminus \{ v_{\ell} \} \subset F_{j(\ell)}$.  If $B_{r} = \{n\}$, without loss of generality we can say that $B_{r} = \{i_{r}\}$. Then $\min\{j_{r}, k_{r} \}  = n-p$ for some $p \geq 1$.  Let $F_{j(\ell)}$ be defined as 
$$
F_{k} \setminus \{v_{r} \} \cup \{(n-p, n-p, n-p),(n-p + 1, n-p + 1, n-p + 1), \ldots , (n,n,n)\}.
$$
Then $\dim F_{j(\ell)} > \dim F_{k}$, so $j(\ell) < k$ and $F_{k} \setminus \{ v_{r} \} \subset F_{j(\ell)}$.  

Now, let $\ell < r$. Either $|B_{\ell}| = 1$ or $|B_{\ell}| = 2$.    (Since $F_{k}$ is a facet and satisfies P3,  $|B_{\ell}| < 3$).  If $|B_{\ell}|  = 2$ we can assume $B_{\ell} = \{i_{\ell}, j_{\ell} \}$.  If $\min\{ j_{\ell + 1} - j_{\ell},  k_{\ell + 1} - k_{\ell} \} = 1$, then let
$$
F_{j(\ell)} = F_{k} \setminus \{v_{\ell} \} \cup \{ (i_{\ell} - 1, j_{\ell}, k_{\ell}) \}.
$$
Then $\sigma(F_{j(\ell)}) < \sigma(F_{k})$, with $F_{k} \setminus \{v_{\ell} \} \subset F_{j(\ell)}$ and $\dim F_{k} = \dim F_{j(\ell)}$, so $j(\ell) < k$.  If $\min\{ j_{\ell + 1} - j_{\ell},  k_{\ell + 1} - k_{\ell} \} > 1$ then $i_{\ell + 1} - i_{\ell} = 1$ because $F_{k}$ is a facet and satisfies P3.  Then let 
$$
F_{j(\ell)} = F_{k} \setminus \{v_{\ell} \} \cup \{ (i_{\ell} - 1, j_{\ell} , k_{\ell} ), (i_{\ell} , j_{\ell} + 1, k_{\ell} + 1) \}
$$
and again $\dim F_{j(\ell)} > \dim F_{k}$, so $j(\ell) < k$ and $F_{k} \setminus \{ v_{\ell} \} \subset F_{j(\ell)}$.  

If $|B_{\ell} |= 1$, then we can assume $B_{\ell} = \{i_{\ell} \}$.  If $ i_{\ell + 1} - i_{\ell} > 1$, then 
$$
F_{j(\ell)} = F_{k} \setminus \{v_{\ell} \} \cup  \{ (i_{\ell} - 1, j_{\ell}, k_{\ell}) \} 
$$
satisfies P3 (because $F_{k}$ does), $\dim F_{k} = \dim F_{j(\ell)}$ and $\sigma(F_{j(\ell)}) < \sigma(F_{k})$,  so $j(\ell) < k$ and $F_{k} \setminus \{ v_{\ell} \} \subset F_{j(\ell)}$. If $i_{\ell + 1} - i_{\ell} = 1$ and $\min \{j_{\ell + 1} - j_{\ell}, k_{\ell + 1} - k_{\ell} \} > 1$, then 
$$
F_{j(\ell)} = F_{k} \setminus \{v_{\ell} \} \cup \{ (i_{\ell} - 1, j_{\ell}-1, k_{\ell}-1), (i_{\ell}, j_{\ell}, k_{\ell}) \}
$$
satisfies $\dim F_{k} < \dim F_{j(\ell)}$ so $j(\ell) < k$ and $F_{k} \setminus \{v_{\ell} \} \subset F_{j(\ell)}$. So whenever $B_{\ell}  \neq \emptyset$ for all $\ell \in [r]$, $F_{k}$ attaches along its entire boundary in the shelling order $\mathcal{O}$ and is a homology facet. 

For the converse, assume that $F_{k} = \{v_{1}, \ldots , v_{r} \}$ is a homology facet.  We wish to show that $B_{\ell} \neq \emptyset$ for all $\ell \in [r]$.  Assume by way of contradiction that there exists $\ell \in [r]$ where $B_{\ell} = \emptyset$.  Since $F_{k}$ is a homology facet, $F_{k} \setminus \{ v_{\ell} \} \subset F_{j(\ell)}$ for some $j(\ell) < k$.  If $\dim F_{k} = \dim F_{j(\ell)}$, then 
$$
F_{j(\ell)} = F_{k} \setminus \{ v_{\ell} \} \cup \{ v'_{\ell} \} 
$$ 
for some $v'_{\ell} \neq v_{\ell}$, and  $\sigma(F_{j(\ell)}) < \sigma(F_{k})$.  Then since the only entries in the sequences $\sigma(F_{j(\ell)})$ and $ \sigma(F_{k})$ that are different come from $v_{\ell}$ and $v'_{\ell}$, one of the three inequalities $(i) \  i'_{\ell} < i_{\ell}$, $(ii) \ j'_{\ell} < j_{\ell}$, or $(iii) \ k'_{\ell} < k_{\ell}$ must be true.  If $i'_{\ell} < i_{\ell} = i_{\ell - 1} + 1$, then this is a contradiction because $i'_{\ell} > i_{\ell - 1}$.  The same contradiction arises if inequalities $(ii)$ or $(iii)$ hold.  

If $\dim F_{j(\ell)} > \dim F_{k}$, then 
$$
F_{j(\ell)} = F_{k} \setminus \{ v_{\ell} \} \cup \{ v_{a, 1}, \ldots , v_{a, d } \} 
$$
where $ d \geq 2$.  
First consider the sub-case where $\ell = r$.    In this instance, $n \in \{ i_{r}, j_{r}, k_{r} \}$.  Without loss of generality we can say $n = i_{r}$.   Since $B_{r} = \emptyset$, $i_{r-1} = n-1$.  Then we must have $n-1 < i_{a, 1} < i_{a, 2}$ and $i_{a, 1} = n$, but since $n$ is the maximum index allowed, this is a contradiction.  Next, consider the sub-case where $\ell < r$.  Then since $F_{k}$ satisfies P3, $\min\{i_{\ell + 1} -i_{\ell}, j_{\ell + 1} -j_{\ell}, k_{\ell + 1} -k_{\ell} \} = 1$.  Without loss of generality, we assume $i_{\ell + 1} - i_{\ell} = 1$. Since $B_{\ell} = \emptyset$, $i_{\ell} - i_{\ell-1} = 1$.  But we must have $i_{\ell -1} < i_{a,1} < i_{a, 2} < i_{\ell} + 1$, which is impossible.  Therefore we have also arrived at a contradiction when $\dim F_{j(\ell)} > \dim F_{k}$.  So when $F_{k}$ is a homology facet, $B_{\ell} \neq \emptyset$ for  all $\ell \in [r]$.  
\end{proof}





\section{Generating Functions that Count the Homology Facets}\label{sec:gen_functions}

In this section we complete our new proof of (\ref{DixID}) by showing the homology facets of $\Delta(n)$ can be counted using generating functions and an application of MacMahon's Master Theorem (Theorem \ref{master}).  For our argument, is useful to consider two families of homology facets, which we now define.

Denote the homology facets of dimension $d$ in $\Delta(n)$ as $H_{d}(\Delta(n))$. We can divide $H_{d}(\Delta(n))$ into two families $X_{d}(n)$ and $Y_{d}(n)$, where $X_{d}(n)$ is the set
$$
\{F  = \{ v_{1}, \ldots , v_{d+ 1} \}  |  a_{d+1} < n \text{ for   some} \ a_{d+1} \in \{ i_{d+1}, j_{d+1}, k_{d+1} \} \} 
$$
and 
$$
Y_{d}(n) = \{F  = \{ v_{1}, \ldots , v_{d+1} \}  | F = G  \in X_{d-1}(\Delta(n-1)) \} \cup \{ (n, n, n)\} \}.
$$

We let  $X(n)$ (respectively $Y(n)$) denote the set $\bigcup_{d = 0}^{n} X_{d}(n)$ and $X$ (respectively $Y$) denote the set $\bigcup_{n = 1}^{\infty} X(n)$.



\begin{lemma}\label{lemma:two_fams}
For $0 \leq d \leq n-1$, $H_{d}(\Delta(n)) = X_{d}(n) \sqcup Y_{d}(n)$. 
\end{lemma}

\begin{proof}
The facets in $X_{d}(n)$ all satisfy $B_{\ell} \neq \emptyset$ (see Definition \ref{down_twist} for a reminder of this notation) for $0 \leq \ell \leq d+1$ by assumption, and by adding the vertex $(n, n, n)$ to a face in $X_{d-1}(n)$, we see that $B_{d + 1} \neq \emptyset$. Therefore every facet in $X_{d}(n)$ and $Y_{d}(n)$ is a homology facet (Lemma \ref{homology_facets_lemma}).  The two sets are disjoint because no face in $X_{d-1}(n-1)$ has $(i_{d}, j_{d}, k_{d}) = (n-1, n-1, n-1)$ by definition. 
\end{proof}

To complete our new proof of (\ref{DixID}),  we must count $H_{d}(\Delta(n))$ for $d \leq n-1$, and show that their alternating sum gives the required right hand side of Equation (\ref{DixID}). 

\begin{defn}\label{shift_vector_defn}
Given a facet $F \in \Delta(n)$ with $r$ vertices, fix $\lambda^{i}, \lambda^{j}$, and $\lambda^{k}$ to represent the collections of indices sorted by position $(i, j,k)$.   With the collection $\lambda^{i}$ we identify the \emph{shift vector} $\overline{\lambda_{i}}$ as $(i_{1}, i_{2}-i_{1}, \ldots , i_{r}- i_{r -1}, n +1 - i_{r})$, with  $\overline{\lambda_{j}}$  and  $\overline{\lambda_{k}}$ defined similarly.
\end{defn}

\begin{ex}
The facet $\{(1,2,2), (2,4,4)\}$ in $\Delta(5)$ has shift vectors $\overline{\lambda_{i}} = (1, 1, 3)$, $\overline{\lambda_{j}} = \overline{\lambda_{k}} = (2, 2, 1)$.
\end{ex}

\begin{rmk}\label{compostition_rmk}
It is clear that every facet is uniquely identified by its shift vectors, and that the shift vectors for a facet with $r-1$ vertices have $r$ entries: the facet $\{(1,2,2), (2,4,4)\}$ in $X_{1}(4)$ has shift vectors $\overline{\lambda_{i}} = (1, 1, 3)$, $\overline{\lambda_{j}} = \overline{\lambda_{k}} = (2, 2, 1)$. 
\end{rmk}

\begin{lemma}\label{NewGenFunctionLemma}
The generating function for a collection $\{\overline{\lambda_{i}}, \overline{\lambda_{j}}, \overline{\lambda_{k}}\}$ of shift vectors derived from facets in the family $X$ with $r-1$ vertices satisfying $$ \sum_{s_{\ell} \in \overline{\lambda_{i}}} s_{\ell} = n_{1}, \sum_{s_{\ell} \in \overline{\lambda_{j}}} s_{\ell} = n_{2} \ \ \text{and}  \sum_{s_{\ell} \in \overline{\lambda_{k}}} s_{\ell} = n_{3} $$ is \[g_r(x,y,z) = \sum_{n_1,n_2,n_3=1}^\infty X(n_1,n_2,n_3)x^{n_1}y^{n_2}z^{n_3} = $$ $$ \left(xyz\left(\frac{1-xyz}{(1-x)(1-y)(1-z)}-1\right)\right)^{r}.  \]
\end{lemma}

\begin{proof}
In the statement of Lemma \ref{NewGenFunctionLemma} we allow for the fact that any given collection $C$ of $\{\overline{\lambda_{i}}, \overline{\lambda_{j}}, \overline{\lambda_{k}}\}$ of shift vectors may come from different complexes in the infinite family $\{ \Delta(n) \}_{n = 1}^{\infty}$. Note that for all indices $\ell \in [r]$,  it follows  (1) from Lemma \ref{facets_lemma} that $1 = \min \{ \overline{\lambda_{i, \ell}}, \overline{\lambda_{j, \ell}}, \overline{\lambda_{k,\ell}} \}$ , and (2) from Lemma \ref{homology_facets_lemma} that $1 < \max\{ \overline{\lambda_{i, \ell}}, \overline{\lambda_{j, \ell}}, \overline{\lambda_{k, \ell}} \} $. 

First we employ a sub-generating function which generates all possibilities for a single index $\ell$ in the collection $\{\overline{\lambda_{i}}, \overline{\lambda_{j}}, \overline{\lambda_{k}}\}$, identified as the vector $(\overline{\lambda_{i, \ell}}, \overline{\lambda_{j, \ell}}, \overline{\lambda_{k, \ell}})$
Suppose that $ \overline{\lambda_{i, \ell}}$ and $\overline{\lambda_{j, \ell}}$ and $ \overline{\lambda_{k, \ell}}$ are both not one:  then, we can generate all possibilities for the vector $(\overline{\lambda_{i, \ell}}, \overline{\lambda_{j, \ell}}, \overline{\lambda_{k, \ell}})$ with the function  \begin{equation}\label{casef} f(x,y,z)=x\sum_{\overline{\lambda_{j, \ell}}=2}^\infty\sum_{\overline{\lambda_{j, \ell}}=2}^\infty y^{\overline{\lambda_{j, \ell}}}z^{\overline{\lambda_{k, \ell}}}  = \frac{xy^2z^2}{(1-y)(1-z)}. \end{equation}

If $\overline{\lambda_{j, \ell}} = \overline{\lambda_{k, \ell}}=1$ then all possible vectors $(\overline{\lambda_{i, \ell}}, \overline{\lambda_{j, \ell}}, \overline{\lambda_{k, \ell}})$ are generated by \begin{equation}\label{caseh} h(x,y,z) =xy\sum_{\overline{\lambda_{k, \ell}}=2}^\infty z^{\overline{\lambda_{k, \ell}}} = \frac{xy z^2}{1-z}. \end{equation}
Add the permutations of the functions in (\ref{casef}) and (\ref{caseh}) to obtain  \[P(x,y,z) =f(x,y,z) +f(y,x,z) +f(z,y,x) +h(x,y,z)+h(x,z,y)+h(z,y,x)  \]  \[ =  \frac{xyz(1-xyz -(1-x)(1-y)(1-z))}{(1-x)(1-y)(1-z)} \] \[ =xyz\left(\frac{1-xyz}{(1-x)(1-y)(1-z)}-1\right).\]

Since there are exactly $r$ vectors $(\overline{\lambda_{i, \ell}}, \overline{\lambda_{j, \ell}}, \overline{\lambda_{k, \ell}})$ generated independently, we obtain

\[ g_{r} (x,y,z) =\left(xyz\left(\frac{1-xyz}{(1-x)(1-y)(1-z)}-1\right)\right)^r.\]
\end{proof}

From the generating function we derived in Lemma \ref{NewGenFunctionLemma} for shift vector collections from homology facets in $X$ to we derive the generating function $(X \sqcup Y)(x,y,z)$ for the alternating sum of all shift vectors corresponding to all homology facets - the set $X \cup Y$ (Lemma \ref{lemma:two_fams})- as follows: $$ (X \sqcup Y)(x,y,z) = \sum_{r=1}P(x,y,z)^r(-1)^{r-1} + xyz\sum_{r=1}P(x,y,z)^{r-1}(-1)^{r-1}  = $$
$$ (P(x,y,z)+xyz)\sum_{r=1}^\infty P(x,y,z)^{r-1}(-1)^{r-1} =  $$ $$\frac{xyz(1-xyz)}{(1-x)(1-y)(1-z)}\sum_{r=1}^\infty P(x,y,z)^{r-1}(-1)^{r-1} = $$
$$ \frac{xyz(1-xyz)}{(1-x)(1-y)(1-z)}\left(\frac{1}{1-xyz+\displaystyle{\frac{xyz(1-xyz)}{(1-x)(1-y)(1-z)}}}\right) =  $$ $$\frac{xyz}{(1-x)(1-y)(1-z)+xyz}.
$$

Dixon's identity follows by the usual application of the following theorem: 

\begin{thm}[Master Theorem \cite{MacMahon}]\label{master}  Let $A=(a_{i,j})_{m\times m}$ and let $X = diag(x_1,\ldots, x_m).$  Then \[[x_{1}^{k_1} \cdots x_{m}^{k_{m}}]\prod_{i=1}^m(a_{i,1}x_1+ \cdots +a_{i,m}x_m)^{k_i} = [x_{1}^{k_1} \cdots x_{m}^{k_{m}}]\frac{1}{Det(I-XA)}\] 
\end{thm}
To obtain the identity, set $n_1=n_2=n_3=n$, so we are only considering the diagonal.  We apply Theorem \ref{master} twice, both times setting $m=3$ and $(x_1,x_2,x_3)=(x,y,z).$  Let \[A= \left[\begin{matrix}1 & -1 & 0 \\ 0 & 1 & -1 \\ -1 & 0 & 1 \end{matrix}\right] \quad \text{and} \quad B= \left[\begin{matrix}0 & 1 & -1 \\ -1 & 0 & 1 \\ 1 & -1 & 0 \end{matrix}\right]\]  which gives us $Det(I-XA) = (1-x)(1-y)(1-z) +xyz$ and so \[[x^ny^nz^n](x-y)^n(y-z)^n(x-z)^n = $$ $$ [x^{n-1}y^{n-1}z^{n-1}]\frac{xyz}{(1-x)(1-y)(1-z)+xyz}.\]  Then we transfer the diagonal by using  $Det(I-XB) = 1+xy+xz+yz$ giving us (\ref{DixID}), as required.

\section{Discussion}\label{sec_discussion}

The construction of $\Delta(n)$ generalizes as follows:  define a simplicial complex $\Gamma_{p}(n)$ with vertices given by sequences in $[n]^{p}$ for $p \geq 1$,  and faces given by collections of vertices
$$
\{ (i_{1,1},  \ldots , i_{1, p}), (i_{2,1},  \ldots , i_{2, p}), \ldots , (i_{r,1},  \ldots , i_{r, p}) \} 
$$
satisfying $i_{\ell, a} \in [n]$ for all $\ell \in [r]$ and all $a \in [p]$, and $i_{\ell, a} < i_{(\ell + 1), a }$ for all $\ell \in [r-1]$ and all $a \in [p]$.  Then $\Gamma_{p}(n)$ has face numbers given by 
$$
f_{s-1} = { n \choose s}^{p}
$$ 
for $0 \leq s \leq n$.  

Note that for $p = 1$,  $\Gamma_{p}(n)$ corresponds to the identity
\begin{equation}\label{ECID}
\sum_{k = 0}^{n} (-1)^{k} { n \choose k} = 0, \quad n \geq 1
\end{equation}
which appears as Exercise 1.3-(f) in \emph{Enumerative Combinatorics Volume I} by Richard Stanley \cite{EC}.  $\Gamma_{1}(n)$ is the traditional $n$-simplex $\Delta_{n-1}$ with vertices labeled with the labels $\{1, \ldots , n \}$ and therefore has the homotopy type of point: so by our technique the fact that the left-hand side of (\ref{ECID}) is $(-1)\widetilde{\chi}(\Delta_{n-1})$ is equal to zero is trivial.

For $p = 2$,  $\Gamma_{p}(n)$ corresponds to the identity
\begin{equation}\label{AignerID}
\sum_{k = 0}^{n} (-1)^{k} { n \choose k}^{2} =  \begin{cases} \ \ 0 \quad  \text{if $n$ is odd}, \\ (-1)^{n/2} {n \choose n/2}, \quad \text{if $n$ is even}. \end{cases}. 
\end{equation}
Determining the value of the right-hand side of (\ref{AignerID}) appears as Exercise 5.48 \cite{Aigner}, and can be established easily using either generating functions or a sign-reversing involution. 

We also note that we can view the simplicial complex $\Delta(n)$ as the order complex of the poset $P_{\Delta(n)}$ where the elements of $P_{\Delta(n)}$ are integer triples in $[n]^{3}$ and $(i, j, k) < (i', j', k')$ in $P_{\Delta(n)}$ if and only if  $i < i'$, $j < j'$, and $k < k'$.  Our proof of the shellability of the complex $\Delta(n)$ does not directly use the fact that $\Delta(n)$ is an order complex of a poset, but we may incorporate this view into future work on this problem.  This perspective has been used to study the \emph{poset of proper divisibility} in the recent paper \cite{Divisibility}.  By adding or removing minimal and maximal elements to the poset in hand, one can establish a bijection between $\Gamma_{p}(n)$ and a subset of the posets studied in \cite{Divisibility}.  The shelling order in \cite{Divisibility} uses techniques of shelling order complexes that generalize sufficiently to show that $\Gamma_{p}(n)$ will be shellable, non-pure, and disconnected for all $p$ and all $n$.  

\section{Future Work}\label{sec_futurework}


A curious fact is that the results in \cite{Divisibility} demonstrate that not only is $\Gamma_{p}(n)$ shellable for all $p$ and for all $n$, but also that identifying the homology facets for any fixed $p$ and $n$ for $\Gamma_{p}(n)$ is possible via techniques from \cite{NonPure1}.  It is also interesting that while the order complexes in \cite{Divisibility} that are in bijection with the members of the family $\Gamma_{p}(n)$ have distinct and interesting topological properties that hold for some ordered pairs $(p,n)$ and not for others, there are no results given relating the number of homology facets to any identities such as (\ref{DixID}).  

This is actually not surprising as counting the homology facets in a ``closed-form" fashion is known to be impossible for arbitrary $n$ and $p$ from past studies of families of identities in other enumerative disciplines.  

For example, in \cite{DeBruijn}, it is shown that for $p \geq 4$, the alternating sum of the face numbers of $\Gamma_{p}(n)$ does not have a ``closed form" for general $n$-which in this case simply means that there is no general formula as a function of $n$ and $p$.  In \cite{Calkin}, citing a personal communication with H. Wilf, N. Calkin mentions that it is not possible to write the unsigned sum
$$
\sum_{k = 0}^{n} {n \choose k}^{p}
$$ 
as a fixed number of hypergeometric terms (recall that we pointed out in the introduction that the identity (\ref{DixID}) belongs to a much more general family of hypergeometric identities).

To finalize the new proof, we appealed to Theorem \ref{master}, tying topological combinatorics back to familiar generating function techniques.  We believe this connection may deserve further study and may lead to a deeper understanding of how different types of enumeration are connected via topological and geometric objects.  In particular: 

\begin{itemize}

\item Studying  $\Gamma_{p}(n)$ for $p \geq 4$ may shed light on the interplay of the mechanics underlying 

\begin{itemize}

\item techniques for studying hypergeometric series,

\item generating function techniques, and

\item asymptotic counting techniques, which are used in the analysis of alternating sums of powers of binomial coefficients in \cite{DeBruijn}.

\end{itemize}


\item It is possible that the combinatorics of the homology facets for dimension $r \geq 1$ as $n$ is allowed to vary may also lead to  interesting formulae or allow us to observe interesting asymptotic behavior. Again, we hope this may lead to a greater understanding of the failure of the aforementioned enumerative techniques to find closed formulae for many alternating sums of binomial coefficients in past efforts. We note that we computed the number of homology facets of $\Delta(n)$ for each dimension $r$ for $n \leq 7$, and neither the sequence of the number of homology facets for increasing $r$ and fixed $n$ nor the sequence generated by fixing $r$ and incrementing $n$ appear in the OEIS \cite{OEIS}. 

\end{itemize}

\section{Acknowledgements}
 
We thank most of all Patricia Hersh, the P.h. D. advisor of R.D. for suggesting this precise formulation of the problem of reproving the identity (\ref{DixID}).  We thank Noam Elkies for suggesting the challenge of proving the identity (\ref{DixID}) using ideas from topology to Patricia Hersh.  We thank Dennis Stanton for sharing his expertise and historical knowledge regarding the identities (\ref{DixID}) and (\ref{3F2}).  We thank Richard Ehrenborg for alerting us to the work in \cite{Calkin} discussing the study of the identities associated with the reduced Euler characteristic of $\Gamma_{p}(n)$.  We thank Jeremy Dewar for assistance in computing the number of homology facets of each dimension for $n \leq 7$. We thank Russ Woodroofe for alerting us to the reference \cite{Divisibility} and for many thought-provoking discussions.


\begin{thebibliography}{}

\bibitem{Aigner} M.~Aigner.  \emph{A Course in Enumeration}.  Berlin:  Springer-Verlag, 2007.  

\bibitem{Borsuk} K.~Borsuk. On the topology of retracts. \emph{Annals of Mathematics}. \textbf{48(4)} (1947): 1082--1094. 

\bibitem{Bailey} W.~N.~ Bailey.  \emph{Generalized Hypergeometric Series.} Cambridge Tracts in Mathematics and Mathematical Physics {\bf 32}, Cambridge University Press, 1935.

\bibitem{BilleraBjorner} L.~J.~Billera and A.~Bj\"{o}rner.  Face numbers of polytopes and complexes.   In: \textit{Handbook of Discrete and Computational Geometry}. Editors: J.~E.~Goodman and J.~O'Rourke. Boca Raton: CRC Press (1997) : 291-310. 


\bibitem{bjorner80} A.~Bj\"{o}rner.  Shellable and Cohen-Macaulay partially ordered sets. \textit{Transactions of the American Mathematical Society} \textbf{260} (1980): 159-183.


\bibitem{NonPure1} A.~ Bj\"{o}rner and  M.~Wachs.  Shellable non-pure complexes and posets I.  \textit{Transactions of the American Mathematical Society} \textbf{348} (1996),  no. 4:  1299--1327.


\bibitem{Divisibility} D.~ Bolognini, A.~Macchia, E.~Ventura, and V.~Welker.  The poset of proper divisibility.   arXiv:1511.04558v2

\bibitem{BruggesserMani} H.~Bruggesser and P.~Mani.  Shellable decompositions of cells and spheres.  \emph{Mathematica Scandinavica} \textbf{29} (1971): 197-205.


\bibitem{Calkin} N.~J.~ Calkin.  Factors of sums of powers of binomial coefficients.  \emph{Acta Arithmetica} \textbf{86(1)} (1998): 17--26.

\bibitem{DavidsonThesis} R. Davidson. Some problems in geometric combinatorics and mathematical phylogenetics. Ph. D. Thesis, North Carolina State University, 2014. http://repository.lib.ncsu.edu/ir/bitstream/
1840.16/9472/1/etd.pdf

\bibitem{DeBruijn}  N.~G.~DeBruijn.  \emph{Asymptotic Methods in Analysis}.  New York: Dover, 1981. 

\bibitem{DixonReference} A.~C.~Dixon. On the sum of the cubes of the coefficients in a certain expansion by the binomial theorem.  \emph{Messenger of Mathematics} {\bf 20} (1891): 79-80.

\bibitem{NonPartionableCM} A.~M.~Duval, B.~Goeckner, C.~J.~Klivans, and J~.L~Martin. A non-partitionable Cohen-Macaulay simplicial complex.  arXiv:1504.04279


\bibitem{AT} A.~Hatcher.  \emph{Algebraic Topology}. Cambridge: Cambridge University Press, 2002. 

\bibitem{HershCommunication2013} P.~Hersh, personal communication (March 2013)

\bibitem{MacMahon} P.~MacMahon.  \emph{Combinatory Analysis-Two Volumes in One}.  New York: Chelsea Publishing, 1960. 

\bibitem{McMullen} P.~McMullen.  The maximum number of faces of a convex polytope.  \emph{Mathematika} {\bf 17} (1970): 179-184. 

\bibitem{OEIS} OEIS Foundation Inc. (2011), The On-Line Encyclopedia of Integer Sequences, http://oeis.org.

\bibitem{Poincare} H.~Poincar\'{e}.  Ir Compl\'{e}meant \`{a} l'Anaylsis situs. \emph{Rendiconti del Circolo matematico di Palermo} {\bf 13} (1899): 285-343.  

\bibitem{CharneyDavis} V.~Reiner, D.~Stanton, and V.~Welker.  The Charney-Davis quantity for certain graded posets.  \emph{S\'{e}minaire Lotharingien de Combinatoire} {\bf 50} (2003): Article B50c.
 
 \bibitem{StanleyBalanced} S.~Stanley. Balanced Cohen-Macaulay complexes. \emph{Transactions of the American Mathematical Society} {\bf 249 (1)} (1979):  139--157.
 
 \bibitem{Whitehead} J.H.C.~Whitehead.  Simplicial spaces, nuclei, and $m$-groups.  {Proceedings of the London Mathematical Society} \textbf{45} (1938): 243-327. 

\bibitem{Schlafli} L.~Schl\"{a}fli.   \emph{Theorie der vielfachen Kontinuit\"{a}t}, 1852.  


\bibitem{EC} R.~Stanley. \emph{Enumerative Combinatorics Volume I}, Cambridge University Press, 1997.  

\bibitem{GeoCombiWachs} M.~Wachs.  Poset topology: tools and applications.  In: \emph{Geometric Combinatorics}.  Editors: E.~Miller, V.~Reiner, and B.~Sturmfels.  IAS/Park City Mathematics Series {\bf 13}.  Providence, Rhode Island: American Mathematical Society, 2007.

\bibitem{masterthoerem} Goulden, Ian P. and Jackson, David M, Combinatorial enumeration.  Dover Publications, Inc., Mineola, NY, 2004.

\end{thebibliography}
\end{document}